\documentclass[10pt]{amsart}

\usepackage{amsmath, amsthm, amscd, amssymb, amsfonts, bbm, wasysym ,mathrsfs}
\usepackage[colorlinks,citecolor=red,pagebackref,hypertexnames=false]{hyperref}
\input amssym.def
\input amssym.tex

\newtheorem{thm}{Theorem}[section]
\newtheorem*{thm*}{Theorem}
\newtheorem{cor}[thm]{Corollary}
\newtheorem*{cor*}{Corollary}
\newtheorem{lem}[thm]{Lemma}

\newtheorem*{con*}{Conjecture}
\newtheorem*{prob*}{Problem}

\theoremstyle{definition}
\newtheorem{defn}[thm]{Definition}

\theoremstyle{remark}
\newtheorem{rem}[thm]{Remark}

\begin{document}
\title{Strongly non embeddable metric spaces}
\author{Casey Kelleher}
\address{Department of Mathematics, California Polytechnic State University, San Luis Obispo, CA 93407}
\email{ckellehe@calpoly.edu}
\author{Daniel Miller}
\address{Department of Mathematics, University of Nebraska at Omaha, Omaha, NE 68182}
\email{dkmiller@unomaha.edu}
\author{Trenton Osborn}
\address{Department of Mathematics, Baylor University, Waco, TX 76798}
\email{trenton\_osborn@baylor.edu}
\author{Anthony Weston}
\address{Department of Mathematics and Statistics, Canisius College, Buffalo, NY 14208}
\email{westona@canisius.edu}



\begin{abstract}
Enflo \cite{Enf} constructed a countable metric space that may not be uniformly embedded into any
metric space of positive generalized roundness. Dranishnikov, Gong, Lafforgue and Yu \cite{Dra}
modified Enflo's example to construct a locally finite metric space that may not be coarsely embedded into any
Hilbert space. In this paper we meld these two examples into one simpler construction.
The outcome is a locally finite metric space $(\mathfrak{Z}, \zeta)$ which is
strongly non embeddable in the sense that it may not be embedded uniformly or coarsely
into any metric space of non zero generalized roundness. Moreover, we
show that both types of embedding may be obstructed by a common recursive principle.
It follows from our construction that any metric space which is
Lipschitz universal for all locally finite metric spaces may not be embedded uniformly or coarsely
into any metric space of non zero generalized roundness. Our construction is then adapted to show that the
group $\mathbb{Z}_\omega=\bigoplus_{\aleph_0}\mathbb{Z}$ admits a Cayley graph which may not be coarsely
embedded into any metric space of non zero generalized roundness. Finally, for each $p \geq 0$ and each
locally finite metric space $(Z,d)$, we prove the existence of a Lipschitz injection $f : Z \to \ell_{p}$.
\end{abstract}
\maketitle

\section{Introduction}

The notion of the generalized roundness of a metric space was introduced by Enflo \cite{Enf} as a means
to show that not every separable metric space may be uniformly embedded into a Hilbert space.
The crux of Enflo's proof was the construction of a countable metric space that may not be uniformly
embedded into any metric space of positive generalized roundness. Enflo then noted that Hilbert
space has generalized roundness $2$ in order to draw the desired conclusion. As in \cite{Enf}, the
terms uniform embedding and generalized roundness are defined in the following way.

\begin{defn}
A \textit{uniform embedding} of a metric space $(X,d)$ into a metric space $(Y,\varsigma)$
is an injection $f : X \to Y$ such that both $f$ and $f^{-1}$ are uniformly continuous,
where the domain of definition of $f^{-1}$ is its natural domain as a subspace of $(Y,\varsigma)$.
\end{defn}

\begin{defn}\label{grdef}
The \textit{generalized roundness} $\mathfrak{q}(X)$ of a metric space $(X,d)$ is the supremum
of the set of all $p \geq 0$ that satisfy the following property: For all integers $n \geq 2$ and all
choices of (not necessarily distinct) points $x_{1}, \ldots, x_{n}, y_{1}, \ldots, y_{n} \in X$,
\begin{eqnarray}\label{ONE}
\sum\limits_{1 \leq i < j \leq n} \bigl\{ d(x_{i},x_{j})^{p} + d(y_{i},y_{j})^{p} \bigl\}
& \leq & \sum\limits_{1 \leq i,j \leq n} d(x_{i},y_{j})^{p}.
\end{eqnarray}
The configuration of points $(\mathscr{X},\mathscr{Y})$ that underlies (\ref{ONE}), where $\mathscr{X}$ denotes
$x_{1}, \ldots, x_{n}$ and $\mathscr{Y}$ denotes $y_{}, \ldots, y_{n}$, is called a \textit{double-$n$-simplex}.
A pair of the form $(x_{i},x_{j})$ or $(y_{i},y_{j})$, where $i \not= j$, will be called an \textit{edge} in
$(\mathscr{X},\mathscr{Y})$. While a pair of the form $(x_{i},y_{j})$, with no restriction on $i$ and $j$,
will be called a \textit{connecting line} in $(\mathscr{X},\mathscr{Y})$.
\end{defn}

\noindent Metric spaces of generalized roundness $0$ may have very rigid uniform structures. For example, the metric
space constructed in the proof of \cite[Theorem 2.1]{Enf} may not be uniformly embedded into any metric space of
positive generalized roundness. Dranishnikov, Gong, Lafforgue and Yu \cite[Proposition 6.3]{Dra} modified Enflo's
example to construct a metric space that may not be coarsely embedded into any Hilbert space. A simple
check shows that this modified metric space also has generalized roundness $0$. It follows from these examples that
the quality of a metric space having generalized roundness $0$ may induce unique properties on the given space, not
only on the fine scale of uniform structure but also on the global scale of coarse geometry. The notion of a coarse
embedding was introduced by Gromov \cite{Gro}.

\begin{defn}\label{coarse}
A map $f$ from a metric space $(X,d)$ into a metric space $(Y,\varsigma)$ is a 
{\it coarse embedding} if there exist two non decreasing functions ${\rho}_{1}, \rho_{2}: [0, {\infty}) \to [0, {\infty})$ such that
\begin{enumerate}
\item ${\rho}_{1}(d(x, y)) \le \varsigma(f(x), f(y)) \le {\rho}_{2}(d(x, y))$ for all $x, y\in X$, and

\item $\lim_{t\to\infty}{\rho}_{1}(t) = \infty$.
\end{enumerate}
\end{defn}

\noindent Although the focus of this paper is squarely on uniform and coarse embeddings we will have the occasion
to refer to Lipschitz embeddings. A \textit{Lipschitz embedding} of a metric space $(X,d)$ into a metric space
$(Y,\varsigma)$ is an injection $f : X \to Y$ such that both $f$ and $f^{-1}$ are Lipschitz
maps of order $1$, where the domain of definition of $f^{-1}$ is its natural domain as a subspace of $(Y,\varsigma)$.
Our interest in Lipschitz embeddings stems from their quality of being both uniform and
coarse embeddings. We further recall that a metric space $(Y,\varsigma)$ is \textit{Lipschitz universal for a class}
$\mathscr{U}$ of metric spaces if every $(X,d) \in \mathscr{U}$ admits a Lipschitz embedding into $Y$.

In this paper we meld the examples of Enflo \cite[Theorem 2.1]{Enf} and Dranishnikov \textit{et al}.\
\cite[Proposition 6.3]{Dra} into one simpler construction. The outcome is a locally finite metric space
$(\mathfrak{Z}, \zeta)$ which is strongly non embeddable in the sense that it may not be embedded uniformly or coarsely
into any metric space of non zero generalized roundness (Theorems \ref{theorem2} and \ref{theorem3}).
Our approach shows that both types of embedding may be obstructed by a common recursive principle (Theorem \ref{theorem1}).
It follows that any metric space which is Lipschitz
universal for all locally finite metric spaces is neither uniformly nor coarsely embeddable into any metric
space of non zero generalized roundness (Corollary \ref{lulfm}). The main construction is then adapted to show that the group
$\mathbb{Z}_\omega=\bigoplus_{\aleph_0}\mathbb{Z}$ admits a Cayley graph which may not be coarsely embedded
into any metric space of non zero generalized roundness. Remark \ref{noboots} shows that it is not possible
to construct an analogue of the metric space $(\mathfrak{Z}, \zeta)$ that has positive generalized roundness.
Finally, for each $p \geq 0$ and each locally finite metric space $(Z,d)$, we prove the existence of a Lipschitz
injection $Z \to \ell_{p}$ (Theorem \ref{discrete}). This indicates that the two types of embeddings under
consideration cannot be weakened if one wants to maintain an obstruction to embedding into metric spaces of
non zero generalized roundness.

To close out this section, we note that examples of metric spaces of non zero generalized roundness include
all real normed spaces that are linearly isometric to a subspace of some $L_{p}$-space with $0 < p \leq 2$,
all metric trees, all ultrametric spaces, all Hadamard manifolds, hyperbolic spaces such as $\mathbb{H}_{\mathbb{C}}^{k}$,
and many CAT($0$)-spaces. By way of comparison, provided $p > 2$, $L_{p}$-spaces of dimension at least $3$ are examples of
metric spaces of generalized roundness $0$. The same is true of certain CAT($0$)-spaces as well as the Schatten $p$-classes
$\mathcal{C}_{p}$ provided $0 < p < 2$. The bulk of these examples are discussed more expansively in the
survey paper \cite{Paw}.

\section{The Main Construction}
Let $\mathsf{L}_n$ be the set $\{k n^{-n}:1\leq k \leq n^{2n}\}$ endowed with the cyclic metric
$\chi(x,y)=\min{\{|x-y|,n^n-|x-y|\}}$. Let $\mathsf{M}_n$ be the cartesian product $\Pi_{i=1}^{n^n}\mathsf{L}_n$
endowed with the supremum metric:
\[
d_{\infty}(x,y)=\max_{1\leq i\leq n^n} \chi(x_{i},y_{i}),
\]
where in this case $x=( x_1,\cdots ,x_{n^n})$ and $y=( y_1,\cdots ,y_{n^n})$.

\begin{defn}
Let $0\leq |t|<n$ and $0 < m < n$. If $x=( x_i ) $ and $y=( y_i )$, then we call
$\{x,y\}\subset \mathsf{M}_n$ a \textit{$(t,m)$-pair} if the following two conditions hold:
\begin{enumerate}
\item $\chi(x_i,y_i)=2^t$ whenever $x_i\ne y_i$, and
\item the cardinality of the set $\{i:x_i\ne y_i\}$ is $n^m$. 
\end{enumerate}
We denote the set of $(t,m)$-pairs by $\Theta_t^m$.
\end{defn}

\begin{defn}
If $(\mathscr{X},\mathscr{Y})$ is a double-$n$-simplex in $\mathsf{M}_n$ and $0<m<n$, $0\leq|t|<n$,
then we call $(\mathscr{X},\mathscr{Y})$ a \textit{$(t,m)$-simplex} if each connecting line is a $(t,m+1)$-pair,
and each edge is a $(t+1,m)$-pair.
\end{defn}

\noindent For the rest of the paper we will only consider the case where $n$ is even.
Also, note that the $(t,m$)-pairs $\{x,y\}$ and $\{y,x\}$ will be regarded as identical,
as will be the double-$n$-simplices $(\mathscr{X},\mathscr{Y})$ and $(\mathscr{Y},\mathscr{X})$.

\begin{lem}\label{lemma1}
If $0<m<n$, $|t|<n$, there exists a $(t,m)$-simplex in $\mathsf{M}_n$.
\end{lem}

\begin{proof}
Divide the $n^n$ coordinates into $3$ groups, the first two of size $n^{m+1}/2$ and the third of size $n^n-n^{m+1}$.
All coordinates in the third group are $0$ in both the $x_i$'s and the $y_i$'s. For the $x_i$'s, all coordinates in
the second group are $2^{t}$, and the first group is divided up into $n$ subgroups of size $n^m/2$. As $i$ varies,
the $i$th subgroup is filled with $2^{t+1}$, and the other $n-1$ subgroups are filled with $0$.
For the $y_i$'s, just switch the first and second groups. The $x_i$'s and $y_i$'s may be visualized
in the following way:
\[
x_i=\Biggl(\,
	\overbrace{
 		\underbrace{0,\cdots,0}_{n^m/2}
	,\cdots,
		\underbrace{2^{t+1},\cdots,2^{t+1}}_{n^m/2}
	,\cdots,
		\underbrace{0,\cdots,0}_{n^m/2}}^{n^{m+1}/2},
	\overbrace{2^t,\cdots,2^t}^{n^{m+1}/2},
	\overbrace{0,\cdots,0}^{n^n-n^{m+1}}
\,\Biggl),
\]
\[
y_i=\Biggl(\,
	\overbrace{2^t,\cdots,2^t}^{n^{m+1}/2},
	\overbrace{
 		\underbrace{0,\cdots,0}_{n^m/2}
	,\cdots,
		\underbrace{2^{t+1},\cdots,2^{t+1}}_{n^m/2}
	,\cdots,
		\underbrace{0,\cdots,0}_{n^m/2}}^{n^{m+1}/2}
	,\overbrace{0,\cdots,0}^{n^n-n^{m+1}}
\,\Biggl).
\]
It is easy to verify that each $x_i$ and $y_j$ all differ in exactly $n^{m+1}$ coordinates by $2^t$,
while each $x_i$ and $x_j$ differ in exactly $n^m$ coordinates by $2^{t+1}$.
\end{proof}

\begin{lem}\label{lemma2}
The isometry group of $\mathsf{M}_n$ acts transitively on $\Theta_t^m$ whenever $m\geq 1$.
\end{lem}

\begin{proof}
First note that any rearrangement of coordinates is an isometry of $\mathsf{M}_{n}$. Given $(t,m)$-pairs
$\{x,y\}$ and $\{u,v\}$, first rearrange the coordinates of $x$ and $y$ so that $\{i:x_i=y_i\}=\{i:u_i=v_i\}$.
This can be done because of condition (2) in the definition of a $(t,m)$-pair. One can then rotate the coordinates
of $x$ and $y$ indexed by $\{i:x_i=y_i\}$ so that $x_i=y_i\Rightarrow x_i=u_i$. Given $i$ such that $x_i\ne y_i$,
flip $x_i$ and $y_i$ as necessary to ensure that $(x_i,y_i)$ and $(u_i,v_i)$ have the same orientation in
$\mathsf{L}_{n}$, and then rotate the $i$th coordinate in $x_i$ and $y_i$ so that $x_i=u_i$, $y_i=v_i$.
This reflection can be carried out because $n^n$ is even. Clearly the operations performed above constitute
an isometry of $\mathsf{M}_n$.
\end{proof}

\noindent Let $N_{(t,m)}$ be the number of $(t,m)$-pairs in $\mathsf{M}_n$. It is clear that there is a natural
number $L_{(t,m)}$ such that any $(t,m)$-pair in $\mathsf{M}_n$ is a connecting line in $L_{(t,m)}$ different
ways in $(t,m-1)$-simplices. Similarly, there exists a natural number $K_{(t,m)}$ such that any $(t,m)$-pair
is an edge in $K_{(t,m)}$ different ways in $(t-1,m)$-simplices.

Let $S_{(t,m)}$ be the number of distinct $(t,m)$-simplices. Any $(t+1,m)$-pair may be embedded as an edge
within a fixed $(t,m)$-simplex in $n(n-1)$ different ways, so $S_{(t,m)} n(n-1)=N_{(t+1,m)}K_{(t+1,m)}$.
Similarly, any $(t,m+1)$-pair may be embedded as a connecting line within a fixed $(m,t)$-simplex in $n^2$
different ways, so $S_{(t,m)} n^2=N_{(t,m+1)}L_{(t,m+1)}$. This yields the following expression:
\begin{align}\label{zwie}
\frac{L_{(t,m+1)}}{K_{(t+1,m)}}&=\left(\frac{n}{n-1}\right)\frac{N_{(t+1,m)}}{N_{(t,m+1)}},
\end{align}
which holds whenever $0<m<n-1$ and $-n < t <n-1$. It is worth noting that $N_{(t,m)}$ depends on $n$ as well
as $t$ and $m$. The context will make the value $n$ clear. To simplify later expressions we will denote $f(x)$ by
$x^f$ and $d(x,y)^{p}$ by $d^{p}(x,y)$ for the rest of this paper.

\begin{thm}\label{theorem1}
Let $(B,d)$ be a metric space of finite generalized roundness $\mathfrak{q}(B)$ and let $f:\mathsf{M}_{n+2}\to B$ be a function.
Provided $0 \leq p \leq \mathfrak{q}(B)$ and $-n \leq t \leq 0$, the following inequality holds:
\begin{eqnarray*}
\frac{1}{N_{(t,n+1)}}\sum_{\{x,y\}\in \Theta_t^{n+1}}d^p (x^f,y^f) & \geq &
\left(1-\frac{1}{n+2}\right)^n\frac{1}{N_{(t+n,1)}}\sum_{\{u,v\}\in \Theta_{t+n}^1}d^p (u^f,v^f).
\end{eqnarray*}
In the event that $\mathfrak{q}(B) = \infty$, the above inequality holds for all $p \geq 0$ and $-n \leq t \leq 0$.
\end{thm}

\begin{proof} Consider any non negative real number $p$ that does not exceed $\mathfrak{q}(B)$.
We first note that if $0 < m <n+1$ and if $(\mathscr{X},\mathscr{Y})$ is any $(t,m)$-simplex in $\mathsf{M}_{n+2}$, then 
\[
\sum_{i,j} d^p(x_i^f,y_j^f) \geq \sum_{i <j} d^p(x_i^f,x_j^f)+\sum_{i <j} d^p(y_i^f,y_j^f)
\]
because (\ref{ONE}) holds for any exponent that does not exceed $\mathfrak{q}(B)$ by \cite[Corollary 2.5]{Ltw}.

By adding up all such inequalities for a fixed $m$, $0 < m<n+1$, we obtain
\[
L_{(t,m+1)} \sum_{\{x,y\}\in \Theta_t^{m+1}} d^p(x^f,y^f) \geq K_{(t+1,m)} \sum_{\{u,v\}\in \Theta_{t+1}^m}d^p(u^f,v^f).
\]
Applying (\ref{zwie}), we see that
\[
\frac{1}{N_{(t,m+1)}}\sum_{\{x,y\}\in \Theta_t^{m+1}}d^p(x^f,y^f) \geq
\left(1-\frac{1}{n+2}\right)\frac{1}{N_{(t+1,m)}}\sum_{\{u,v\}\in \Theta_{t+1}^m}d^p(u^f,v^f).
\]
Start with $m=n$, and iterate this inequality $n$ times to arrive at the desired result.
The case where $\mathfrak{q}(B) = \infty$ follows from \cite[Corollary 2.5]{Ltw} in the same way.
\end{proof}

\noindent We conclude this section by describing the locally finite metric space $(\mathfrak{Z},\zeta)$ that will
be the focus of the subsequent sections. Theorems \ref{theorem2} and \ref{theorem3} determine that this metric space
may not be embedded uniformly or coarsely into any metric space of non zero generalized roundness.

\begin{defn}\label{enflowed}
Let $\mathfrak{Z}=\coprod_{n\in 2\mathbb{N}}\mathsf{M}_n$, with metric $\zeta$ such that $\zeta$ restricted
to any $\mathsf{M}_n$ gives the original metric on $\mathsf{M}_n$, and such that $\zeta(x,y)=4^{m+n}$ if
$x\in \mathsf{M}_m$ and $y\in \mathsf{M}_n$.
The metric space $(\mathfrak{Z}, \zeta)$ is locally finite by the nature
of its definition.
\end{defn}

\section{The Obstruction to Coarse Embeddings}

\begin{thm}\label{theorem2}
The locally finite metric space $(\mathfrak{Z},\zeta)$ does not coarsely embed into any metric
space of non zero generalized roundness.
\end{thm}

\begin{proof}
We first suppose that there exists a metric space $(B,d)$ with $\mathfrak{q}(B)=p > 0$ and a coarse
embedding $f:\mathfrak{Z}\to B$. By the definition of a coarse embedding, there exist non-decreasing
functions $\rho_1,\rho_2:[0,\infty)\to[0,\infty)$ such that
\[
\rho_1(\zeta(x,y))\leq d(x^f,y^f) \leq \rho_2 (\zeta(x,y))
\]
for all $x,y\in \mathfrak{Z}$, and such that $\lim_{r\to \infty} \rho_1 (r)=\infty$.
Then select a positive even integer $n$ together with a real number $\alpha$ such that
$\rho_1 (2^n)\geq \alpha \rho_2 (1)$, and $\alpha^p e^{-1} >1$. Now consider $\mathsf{M}_{n+2}\subset \mathfrak{Z}$,
and set $t=0$. By Theorem \ref{theorem1},
\begin{align} \notag
\frac{1}{N_{(0,n+1)}}\sum_{\{x,y\}\in \Theta_0^{n+1}}d^p(x^f,y^f)& \geq
\left(1-\frac{1}{n+2}\right)^n\frac{1}{N_{(n,1)}}\sum_{\{u,v\}\in \Theta_n^1}d^p(u^f,v^f)\\ \notag
&\geq e^{-1}\rho_1^p (2^n) \\ \notag &\geq\alpha^p e^{-1} \rho_2^p (1).
\end{align}

On the other hand, we have that

\[
\frac{1}{N_{(0,n+1)}}\sum_{\{x,y\}\in \Theta_0^{n+1}}d^p(x^f,y^f)\leq \rho_2^p (1),
\]
which clearly contradicts $\alpha^pe^{-1}>1$. In the case where $\mathfrak{q}(B) = \infty$, we may
argue as above by using any positive real number $p$ in place of $\mathfrak{q}(B)$. This is due to
the second assertion of Theorem \ref{theorem1}.
\end{proof}

\noindent It is worth noting that if one is only interested in locally finite metric spaces that do not
embed coarsely into Hilbert space, then there is a simpler construction than that of $(\mathfrak{Z},\zeta)$.
Let $(X_{n})$ be an infinite sequence of finite, connected, $k$-regular graphs which is an expanding family,
and let $X$ be the disjoint union of the $X_{n}$'s endowed with a metric that induces the original metric
on each $X_{n}$, then $X$ does not coarsely embed into any Hilbert space. This example is due to Gromov \cite{Mgr}.

In view of the emphasis that is often placed on locally finite metric spaces in coarse geometry it makes
sense to isolate the following natural definition.

\begin{defn}
A metric space $(X,d)$ is said to be a \textit{universal coarse embedding space} if every locally finite
metric space coarsely embeds into $X$.
\end{defn}

\noindent Examples of universal coarse embedding spaces include the Banach space $C[0,1]$ and the Urysohn space $\mathbb{U}$.
We note that background information on the Urysohn space may be found in Pestov \cite[Chapter 5]{Pes}. Theorem \ref{theorem2}
automatically implies that no universal coarse embedding space can have non zero generalized roundness. This
mirrors the corresponding result for uniform embeddings \cite[Theorem 2.1]{Enf}.

\begin{cor}\label{uces}
No universal coarse embedding space has non zero generalized roundness. In particular,
metric subspaces of $L_{p}$-spaces, where $0 < p \leq 2$, and ultrametric spaces are not
universal coarse embedding spaces.
\end{cor}

\noindent We conclude this section by noting that there is an interesting way to adapt Theorem \ref{theorem2}
to Cayley graphs of certain abelian groups. We provide one such example.

Recall that if $G$ is an abelian group, then for each generating set $\Sigma$, one can consider the Cayley graph
$\Gamma = \mbox{Cay}\,(G,\Sigma)$ of $G$ with respect to $\Sigma$, endowed with the usual graph metric.
In most cases, $\mathfrak{q}(\Gamma)\leq 1$. Simply select $g,h\in \Sigma$ such that $g+h,g-h\notin \Sigma$,
and consider the double-$2$-simplex $(\mathscr{X},\mathscr{Y})$ with $\mathscr{X}=\{0,g\}$, $\mathscr{Y}=\{h,g+h\}$.
It is readily checked that the generalized roundness inequality for $(\mathscr{X},\mathscr{Y})$ holds only if $p\leq 1$.

Non zero distances in Cayley graphs are all at least one, so it is not possible for $\mathsf{M}_{n}$ to be
isometric to a subset of any Cayley graph. However, if we define $\mathsf{M}_n^{\ast}$ to be the subset of $\mathsf{M}_n$
consisting of points with integer coordinates, then it is apparent that the proofs of Theorems \ref{theorem1} and \ref{theorem2}
can be adapted to any space that contains isometric copies of $\mathsf{M}_n^{\ast}$ for arbitrarily large $n$. This leads to the
following theorem.

\begin{thm}\label{theorem4}
There exists a generating set $\Sigma$ for $\mathbb{Z}_\omega=\bigoplus_{\aleph_0}\mathbb{Z}$ such 
that the Cayley graph of $\mathbb{Z}_\omega$ with respect to $\Sigma$ contains isometric copies of
$\mathsf{M}_n^{\ast}$ for arbitrarily large $n$. In particular, the Cayley graph $\mbox{Cay}\, (\mathbb{Z}_{\omega},\Sigma)$
does not coarsely embed into any metric space of non zero generalized roundness.
\end{thm}

\begin{proof}
We may rewrite $\mathbb{Z}_\omega$ as $\bigoplus_{n=1}^\infty A_n$, where $A_n=\mathbb{Z}^{n^n}$.
For each $n$, let $\Sigma_n$ be the set of vectors contained in $A_n$ whose coordinates are
either $0$ or $\pm (n^n-1)$. Let $\Sigma$ be $\bigcup_{n=1}^\infty \Sigma_n$ together with
all maps $\omega\to \{0,\pm 1\}$ of finite support. It is easy to see that the restriction
of the graph metric on $\mbox{Cay}\,(\mathbb{Z}_\omega,\Sigma)$ to $\mbox{Cay}\,(A_n,A_n\cap \Sigma)$
yields the usual metric on $\mathsf{M}_n^*$. The proof of the second statement is a straightforward
application of Theorem \ref{theorem2}.
\end{proof}

\begin{rem}
In the case of obstructing coarse embeddings into Hilbert space, it is worth noting that
the second statement of Theorem \ref{theorem4} is similar in spirit to an example of Nowak
\cite[Remark 5 (A)]{Now}. A feature of our approach is that it is more general and it uses
a less sophisticated machinery.
\end{rem}

\section{The Obstruction to Uniform Embeddings}

The metric space $(\mathfrak{Z},\zeta)$ constructed in Definition \ref{enflowed} has enough fine structure
that it also obstructs uniform embeddings into metric spaces of non zero generalized roundness. Once again,
the obstruction is provided by Theorem \ref{theorem1}.

\begin{thm}\label{theorem3}
The locally finite metric space $(\mathfrak{Z},\zeta)$ does not uniformly embed into
any metric space of non zero generalized roundness.
\end{thm}

\begin{proof} We first
suppose that there exists a metric space $(B,d)$ with $\mathfrak{q}(B)=p > 0$ and a uniform embedding $f$
of $\mathfrak{Z}$ into $B$. If we set $t=-n$ and consider $\mathsf{M}_{n+2}\subset\mathfrak{Z}$,
then it follows from Theorem \ref{theorem1} that 
\[
\sup_{\{x,y\}\in \Theta_{-n}^{n+1}}d^p(x^f,y^f)\geq e^{-1}\inf_{\{u,v\}\in \Theta_0^1}d^p(u^f,v^f),
\]
and so
\begin{eqnarray}\label{TWO}
\sup_{\{x,y\}\in \Theta_{-n}^{n+1}} d(x^f,y^f)\geq \sqrt[p]{e^{-1}}\inf_{\{u,v\}\in \Theta_0^1} d(u^f,v^f).
\end{eqnarray}
However, $\zeta(u,v) \geq 1$ for all $\{ u,v \} \in \Theta_{0}^{1}$ and $f$ has a uniformly continuous inverse,
so there must exist an $\varepsilon>0$ such that $\inf_{\{u,v\}\in \Theta_0^1} d(u^f,v^f) \geq \varepsilon$. The
net effect is that the right side of (\ref{TWO}) must be bounded away from zero, which is impossible.
This is because $\zeta(x,y) = 2^{-n}$ for all $\{ x,y \} \in \Theta_{-n}^{n+1}$ and $f$ is uniformly continuous,
providing a contradiction in the limit as $n \rightarrow \infty$. In the case where $\mathfrak{q}(B) = \infty$, we may
argue as above by using any positive real number $p$ in place of $\mathfrak{q}(B)$. This is due to
the second assertion of Theorem \ref{theorem1}.
\end{proof}

\begin{rem}\label{noboots}
The metric space $(\mathfrak{Z},\zeta)$ whose properties are examined in Theorems \ref{theorem2} and
\ref{theorem3} has generalized roundness $p = 0$ together with ample circular symmetry.
It is not possible to adapt this construction to positive values of $p$ in any way whatsoever.
No form of Theorem \ref{theorem2} or Theorem \ref{theorem3} exists for metric spaces of generalized roundness $p>0$.
In other words, for any metric space $(X,d)$ of generalized roundness $p>0$, there exists a metric space of generalized
roundness $q>p$ that is both uniformly and coarsely equivalent to $(X,d)$.
To see this we simply consider the metric transform $\rho (x,y)=d(x,y)^{p/q}$. It is easy to see
that $(X,\rho)$ has generalized roundness $q$, and that the identity map on $X$ is both a uniform
and a coarse equivalence between $(X,d)$ and $(X,\rho)$. It is also worth noting that the uniform continuity of
the inverse of $f$ in the proof of Theorem \ref{theorem3} is absolutely necessary. We illustrate this by showing
that every locally finite metric space admits a Lipschitz injection into $\ell_0$ and a uniformly continuous
injection into $\ell_p$ for each $p>0$.
\end{rem}

\begin{thm}\label{theorem5}
Every locally finite metric space may be Lipschitz injected into $\ell_0$.
\end{thm}

\begin{proof}
Let $(Z,d)$ be a locally finite metric space. We assume that $|Z| > 1$ to avoid a triviality. For each
$x \in Z$, let $x^{g}=\min\{d(x,y): x \not= y \}$. We may then define, for each $n \in \mathbb{N}$, the
(possibly empty) set $A_{n} = \{ x \in Z : 2^{1-n} > x^{g} \}$. Provided $Z \setminus A_{n} \not= \emptyset$,
we may choose an injection $h_{n} : Z \setminus A_{n} \rightarrow (0, \infty)$ because $Z \setminus A_{n}$
is at most countable. This allows us to define a map $f: Z \to \ell_0 : x \mapsto
(x^{f_{1}}, x^{f_{2}}, \ldots )$ by
\[
x^{f_{n}}=\left\{
     \begin{array}{ll}
       0 & \mbox{if } x \in A_{n}, \\
       x^{h_{n}} & \mbox{otherwise.}
     \end{array}
   \right.
\]
Notice that if $x \not= z$, then we may choose $n \in \mathbb{N}$ so that both $x,z \notin A_{n}$.
Then $x^{h_{n}} \not= z^{h_{n}}$ because $h_{n} : Z \setminus A_{n} \rightarrow (0, \infty)$ is an injection.
Thus $f(x) \not= f(y)$, and so it follows that $f: Z \to \ell_0$ is an injection.

For $x,y \in Z$, we now define $\kappa (x,y)= \min\{n\in\mathbb{N}:2^{1-n}\leq \max\{x^{g},y^{g}\}\}$.
Notice that, $\kappa(x,y) \leq \min \{ \kappa(x,x), \kappa(y,y) \}$ for all $x,y \in Z$.
Moreover, if $k = \kappa(x,x)$ and if $n<k$, then $x^{f_{n}}=0$. Thus, given any $x,y\in Z$, we see that
\begin{align} \notag
\|x^f-y^f\|_0&=\sum_{n=1}^\infty \frac{|x^{f_n}-y^{f_n}| \cdot 2^{-n}}{1+|x^{f_n}-y^{f_n}|} \\ \notag
&=\sum_{n=\kappa(x,y)}^\infty \frac{|x^{f_n}-y^{f_n}|\cdot 2^{-n}}{1+|x^{f_n}-y^{f_n}|} \\ \notag
&\leq \sum_{n=\kappa(x,y)}^\infty 2^{-n}\\ \notag
&=2^{1-\kappa(x,y)}\\ \notag
&\leq \max\{x^{g},y^{g}\}\\ \notag
&\leq d(x,y),
\end{align}
thereby establishing that $f$ is a Lipschitz injection of order $1$.
\end{proof}

\begin{thm}\label{theorem6} Let $p > 0$.
Every locally finite metric space admits a uniformly continuous injection into $\ell_p$.
\end{thm}

\begin{proof}
Let $p > 0$ and let $(Z,d), g, h_{n}, f$ and $\kappa$ be as in the proof of Theorem \ref{theorem5}, with the
exception that $x^{h_{n}} \in (0,2^{-n/p})$ when $2^{1-n}\leq x^{g}$. We then evidently have:
\[
\|x^f-y^f\|_{p}^{p}= \sum_{n=\kappa(x,y)}^\infty |x^{f_n}-y^{f_n}|^p
\leq \sum_{n=\kappa(x,y)}^\infty |2^{-n/p}|^p.
\]
The latter sum is $2^{1-\kappa(x,y)}$ as before, so we see that $\|x^f-y^f\|_p\leq d(x,y)^{1/p}$.
\end{proof}

\noindent It turns out that Theorems \ref{theorem5} and \ref{theorem6} are special instances of a
more general phenomenon concerning the existence of Lipschitz injections. Recall that a metric space $(Z,d)$ is
said to be \textit{discrete} if, for each $x \in Z$, there exists a $\delta > 0$ such that $d(x,y) > \delta$
for all $y \not= x$. For example, all locally finite metric spaces are discrete.

\begin{thm}\label{discrete}
Let $(Z,d)$ be a given discrete metric space and suppose that $(Y,\rho)$ is any metric space
that contains a nested sequence of open balls $B_{1} \supset B_{2} \supset B_{3} \supset \ldots$
such that $(1)$ the diameter $t_{n}$ of $B_{n}$ tends to $0$ as $n \rightarrow \infty$, and
$(2)$ $|Z| \leq |B_{n}|$ for all $n \in \mathbb{N}$. Then, $Z$ may be Lipschitz injected into $Y$.
\end{thm}

\begin{proof}
Suppose $(Z,d)$ and $(Y,\rho)$ satisfy the hypotheses of the theorem with $|Z| > 1$.
For each $x \in Z$ and $t > 0$, we define the functions $x^{g} = \inf \{ d(x,y) : y \not= x \}$ and
$t^{h} = \min \{ n \in \mathbb{N} : t_{n} \leq t \}$. By conditions $(1)$ and $(2)$ there must exist a
function $f : Z \to Y$ that maps each $x \in Z$ uniquely into $B_{(x^{g})^{h}}$. For any $x,y \in Z$,
we then have $x^{f}, y^{f} \in B_{(\max \{ x^{g}, y^{g} \})^{h}}$, implying that
$\rho(x^{f}, y^{f}) \leq \max \{ x^{g}, y^{g} \} \leq d(x,y)$. Hence the function $f$ is a
Lipschitz injection of order $1$.
\end{proof}

\begin{rem}
Theorem \ref{discrete} is easily seen to imply Theorems \ref{theorem5} and \ref{theorem6}. In fact,
Theorem \ref{discrete} implies a considerably stronger but less explicit version of Theorem \ref{theorem6}:
For each locally finite metric space $(Z,d)$ and each $p > 0$, there exists a Lipschitz injection $f : Z \to \ell_{p}$.
This is significant in the context of Theorem \ref{theorem2}. For coarse embeddings the behavior of
$f^{-1}$ is governed by the condition on the function $\rho_{1}$ that is described in Definition \ref{coarse}.
The existence of the Lipschitz injections implied by Theorem \ref{discrete}, combined with the fact
that $\ell_{p}$ has positive generalized roundness whenever $0 < p \leq 2$, makes it clear that we
cannot dispense with the function $\rho_{1}$ and expect to maintain the conclusion of Theorem \ref{theorem2}
even if $\rho_{2}(t) = t$ for all $t > 0$.
We further deduce from Theorem \ref{discrete} that any locally finite metric space may be Lipschitz
injected into any metric space that contains a Cauchy sequence of distinct elements.
Indeed, any countable discrete metric space $(Z,d)$ may be Lipschitz injected into any metric space
that contains a Cauchy sequence of distinct elements. For if $\{x_n\}$ is such a Cauchy sequence in
$(Z,d)$ and $m \in \mathbb{N}$, we may let $B_m$ be any open ball of diameter $\leq 1/m$ that contains
all but finitely many of the $x_n$'s.
\end{rem}

\noindent In conjunction with existing theory, Theorems \ref{theorem2} and \ref{theorem3} say something interesting
about the Banach space $c_0$. Recall that $c_{0}$ consists of all real sequences that converge to zero endowed
with the supremum norm. It was shown by Aharoni \cite{Aha} that every separable metric space admits a Lipschitz embedding
into $c_0$. Thus, there exists a Lipschitz embedding $(\mathfrak{Z}, \zeta) \to c_0$.
Now let $B$ be a metric space of non zero generalized roundness. Lipschitz embeddings are both uniform
and coarse embeddings, so if there existed an embedding $c_0\to B$ that was either a uniform
or coarse embedding, then there would exist such an embedding $(\mathfrak{Z}, \zeta) \to B$. However,
this is not possible according to Theorems \ref{theorem2} and \ref{theorem3}. More generally, these remarks
apply to any metric space which is Lipschitz universal for all locally finite metric spaces and not just $c_{0}$.
Interesting metric spaces of this nature include the widely studied Urysohn space $\mathbb{U}$.
We therefore obtain the final result of this paper. We should note, however, that in the case of uniform embeddings,
this result follows from \cite[Theorem 2.1]{Enf} by a minor modification.

\begin{cor}\label{lulfm}
The Banach space $c_0$ is neither uniformly nor coarsely embeddable into any metric space of
non zero generalized roundness. More generally, any metric space which is Lipschitz universal for
all locally finite metric spaces is neither uniformly nor coarsely embeddable into any metric space of
non zero generalized roundness.
\end{cor}

\noindent As already noted, the preceding corollary implies that the Urysohn space $\mathbb{U}$ is neither uniformly nor coarsely
embeddable into any metric space of non zero generalized roundness. It is very interesting to compare this result to
a theorem of Pestov \cite{Vpe}: The Urysohn space $\mathbb{U}$ is neither uniformly nor coarsely embeddable
into any uniformly convex Banach space. Notice, for example, that the Banach space $\ell_{1}$ has non zero generalized roundness
but it is not uniformly convex. On the other hand, the Banach space $\ell_{3}$ is uniformly convex but it does not have non zero
generalized roundness.

\section*{Acknowledgments}

\noindent The research presented in this paper was undertaken and completed at the 2011
Cornell University \textit{Summer Mathematics Institute} (SMI). The authors would like to thank the
Department of Mathematics and the Center for Applied Mathematics at Cornell University
for supporting this project, and the National Science Foundation for its financial support of
the SMI through NSF grant DMS-0739338. The last named author thanks Canisius College for a
faculty summer research fellowship.

\bibliographystyle{amsalpha}

\end{document}